\documentclass[12pt]{amsart}
\usepackage{amsmath,amsthm,amsfonts,amssymb,times,latexsym,mathabx,url}
\usepackage[pdftex]{graphicx}

\newtheorem{theorem}{Theorem}[section]

\newtheorem{lem}[theorem]{Lemma}
\newtheorem{cor}[theorem]{Corollary}

\numberwithin{equation}{section}

\renewcommand{\L}{\mathcal{L}}
\newcommand{\e}{\varepsilon}

\renewcommand{\a}{\alpha}

\newcommand{\N}{\mathbb{N}}
\newcommand{\Z}{\mathbb{Z}}
\renewcommand{\leq}{\leqslant}
\renewcommand{\geq}{\geqslant}
\renewcommand{\d}{\delta}

\renewcommand{\L}{\Lambda}

\makeatletter
\renewcommand{\pmod}[1]{\allowbreak\mkern7mu({\operator@font mod}\,\,#1)}
\makeatother

\newcommand{\be}{\begin{equation}}
\newcommand{\ee}{\end{equation}}

\renewcommand{\a}{\ensuremath{\alpha}}

\renewcommand{\leq}{\leqslant}

\renewcommand{\geq}{\geqslant}

\newcommand{\Q}{\mathbb Q}

\newcommand{\R}{\mathbb R}
\renewcommand{\L}{\Lambda}

\renewcommand{\a}{\alpha}

\newcommand{\g}{\gamma}
\renewcommand{\leq}{\leqslant}
\renewcommand{\geq}{\geqslant}

\begin{document}

\title{Trigonometric series with noninteger harmonics}
\author{Mikhail R. Gabdullin}
\date{}
\address{
Steklov Mathematical Institute,
Gubkina str., 8, Moscow, Russia, 119991}
\email{gabdullin.mikhail@yandex.ru, gabdullin@mi-ras.ru} 

\begin{abstract}
Let $\{c_k\}$ be a nonincreasing sequence of positive numbers  (more general classes of sequences are also considered), and $\a>0$ be not an integer. We find necessary and sufficient conditions for the uniform convergence of the series $\sum_k c_k\sin k^{\a}x$ and $\sum_k c_k\cos k^{\a}x$ on the real line and its bounded subsets.
\end{abstract}

\date{\today}

\maketitle

\section{Introduction} 

Let $\{c_k\}_{k=1}^{\infty}$ be a nonincreasing sequence of positive numbers, and $\a>0$. We investigate necessary and sufficient conditions for the series
\begin{equation}\label{1.1}
\sum_{k=1}^{\infty}c_k\sin k^{\a}x
\end{equation}
and
\begin{equation}\label{1.2}
\sum_{k=1}^{\infty}c_k\cos k^{\a}x
\end{equation}
to converge uniformly on the real line $\R$ and its bounded subsets. 
Trivially, the condition $\sum_kc_k<\infty$ is sufficient, but can one weaken it? This question for the sine series (\ref{1.1}) was considered in several papers. More than a century ago it was shown in \cite{CJ} that in the case $\a=1$ the series (\ref{1.1}) converges uniformly on $[0,\pi]$ (equivalently, on the real line $\R$) if and only if $c_kk\to0$ as $k\to\infty$ (note that, under the monotonicity assumption, $\sum_kc_k<\infty$ implies $c_kk\to0$; see Lemma \ref{lem2.1}). Since then, many generalizations of this theorem have also been obtained, where the same was shown for more general classes of sequences $\{c_k\}$ (that is, the requirement of its monotonicity was relaxed); see, for instance, \cite{Nur}, \cite{Ste}, \cite{Lei}, \cite{Tik07}, and also \cite{Og} for an overview of such results.

In the recent paper \cite{Kes} it was proved that in the case $\a=1/2$ the condition $c_kk\to0$, $k\to\infty$, is necessary and sufficient for the uniform convergence of the series (\ref{1.1}) on $[0,\pi]$, and that in the case $\a=2$ it converges uniformly on $[0,\pi]$ (equivalently, on $\R$) if and only if $\sum_kc_k<\infty$; the case $\a=1/2$ is also discussed in the wonderful book \cite{Sel} (see Problem IV.5.10). As for the case of general $\a$, a nice breakthrough was made in \cite{Og}. The main results of that paper are the following: firstly, for any odd $\a\in\N$, the series (\ref{1.1}) converges uniformly on $\R$ if and only if $c_kk\to0$, $k\to\infty$ (the same is shown for the series $\sum_k\sin f(k)x$, where $f$ is  an odd polynomial with rational coefficients); secondly, for any $\a\in(0,2)$ the same condition is equivalent to the uniform convergence of the series (\ref{1.1}) on any bounded subset of $\R$; and thirdly, for any even $\a\in\N$ its uniform convergence on $\R$ is equivalent to the condition $\sum_kc_k<\infty$. It is also noted in \cite{Og} that the results are valid for the sequences $\{c_k\}$ from RBVS class (rest bounded variation sequences), that is, satisfying the condition
$$\sum_{k\geq l}|c_k-c_{k+1}| \leq Vc_l
$$
for any $l$ with an absolute constant $V$; clearly, any nonincreasing sequence of positive numbers belongs to RBVS class.    
 
In what follows, we consider (unless otherwise stated) the sequences $\{c_k\}$ with the properties
\begin{equation}\label{1.3}
c_k \leq Ac_l, \quad k\geq l, 
\end{equation}
and 
\begin{equation}\label{1.4}
\sum_{k=l}^{2l-1}|c_k-c_{k+1}| \leq Bc_l, \quad l\in\N, 
\end{equation} 
where $A, B > 0$ are some absolute constants; it can be shown that this class of sequences is wider than RBVS class (see \cite{Tik07}). We note that the condition (\ref{1.4}) was introduced in \cite{Tik05} and defines GM class (generale monotone sequences). 
 
Our contribution to this topic is given in the following three theorems. Firstly, for any noninteger $\a>0$  we show that for the uniform convergence of the series (\ref{1.1}) or (\ref{1.2}) on any interval $[a,b]$ with $0<a<b$ one can require something weaker than $c_kk\to0$, $k\to\infty$.

\begin{theorem}\label{th1.1} 
	Let $\{c_k\}$ be a sequence obeying (\ref{1.3}) and (\ref{1.4}). Then 
	
	(a) for any $\a>0$, $\a\notin \N$, there exists $c(\a)\in(0,1)$ such that the condition $\sum_kc_kk^{-c(\a)}<\infty$ is sufficient for the uniform convergence of the series \eqref{1.1} and \eqref{1.2} on $[a,b]$ for any $0<a<b$. Moreover, one can take $c(\a)=\a$ for $\a\in(0,1)$.  	
	
	(b) for any $\a\in(0,1)$ the condition $c_kk^{1-\a}\to0$, $k\to\infty$, is necessary for the convergence of the series \eqref{1.1} or \eqref{1.2} at any point $x\neq0$.  	
\end{theorem}

Note that the condition in the part (a) is clearly weaker than $c_kk=o(1)$. The part (b) means that for $\a\in(0,1)$ the condition $\sum_kc_kk^{-\a}<\infty$ cannot be essentially relaxed: indeed, if for some $\e>0$ the condition $\sum c_kk^{-\a-\e}<\infty$ were sufficient, then the series with $c_k=k^{\a-1+\e/2}$ would converge, which is impossible because of (b). We did not try to optimize the value $c(\a)$ for $\a>1$, and just mention that one can take $c(\a)$ of order $\a^{-2}$ for large $\a$ (see Corollary \ref{cor2.4}).

Secondly, we generalize the mentioned result from \cite{Og} for any noninteger $\a>0$.

\begin{theorem}\label{th1.2} 
Let $\a>0$ be not an integer, and let $\{c_k\}$ be a sequence obeying (\ref{1.3}) and (\ref{1.4}). Then the series \eqref{1.1} converges uniformly on any bounded subset of $\R$ if and only if $c_kk\to0$ as $k\to\infty$.
\end{theorem}

Combining this with the results of \cite{Og}, we obtain criteria for the uniform convergence of \eqref{1.1} on the bounded subsets of the real line for any $\a>0$.

Finally, for rational noninteger $\a>0$ and nonincreasing sequences $\{c_k\}$ we find the criterion for the uniform convergence of the series (\ref{1.1}) on the whole real line.

\begin{theorem}\label{th1.3} 
Let $\a>0$, $\a\in\Q\setminus\N$, and let $\{c_k\}$ be a nonincreasing sequence of positive numbers. Then the series (\ref{1.1}) converges uniformly on $\R$ if and only if $\sum_kc_k < \infty$.  	
\end{theorem}

Recall that the same for even $\a>0$ was shown in \cite{Og}.

\smallskip 

In Section \ref{sec2} we provide some auxiliary results. Sections \ref{sec3}, \ref{sec4}, \ref{sec5} are devoted to the proofs of Theorems \ref{th1.1}, \ref{th1.2}, \ref{th1.3} respectively. Theorem \ref{th1.1} follows from estimates for the exponential sums $\sum_{M\leq n <2M}e^{2\pi in^{\a}x}$; the proof of Theorem \ref{th1.2} relies on it as well, but also needs bounds for the sums $\sum_{M\leq n <2M}\sin{n^{\a}x}$, where $x$ is small with respect to $M$. The proof of Theorem \ref{th1.3} works for any $\a>0$ such that the square-free\footnote{A positive integer $n$ is said to be square-free if it has the form $n=p_1\ldots p_s$, where $p_j$ are distinct primes.} numbers (or at least a large part of them) to the degree $\a$ are linearly independent over $\Q$. It is natural to conjecture that it is so for any noninteger $\a$, but this seems to be known only for noninteger rational ones. We mention that the analogues of Theorems \ref{th1.2} and \ref{th1.3} for the series (\ref{1.2}) are trivial, since the convergence of this series at the point $x=0$  implies that $\sum_kc_k<\infty$.

\bigskip 

\textbf{Notation.} We use Vinogradov's $\ll$ notation: $F\ll G$ (or $G\gg F$) means that there exists a constant $C>0$ such that $|F|\leq CG$. In many cases, this constant is allowed to depend on some parameters; say, the notation $F\ll_{\a,a,b}G$ means that there is a number $C=C(\a,a,b)$ such that $|F|\leq C(\a,a,b)G$. We write $F\asymp G$ if $G\ll F\ll G$. We denote by $\lfloor u\rfloor$ the largest integer not exceeding $u$.

\bigskip 

\textbf{Acknowledgements.} The author would like to thank Sergei Konyagin for introducing him to this topic and Sergey Tikhonov and the anonymous referee for useful remarks. This work was performed at the Steklov International Mathematical Centre and supported by the Ministry of Science and Higher Education of the Russian Federation (agreement no. 075-15-2019-1614).

\section{Auxiliary results}\label{sec2}

In this section we provide several technical results which we will use later. We begin with the following simple lemma.

\begin{lem}\label{lem2.0}
Let $\{u_k\}$ and $\{c_k\}$ be sequences obeying (\ref{1.3}). Then  $\sum_kc_ku_k<\infty$ implies $c_kku_k\to0$ as $k\to\infty$.
\end{lem}	

\begin{proof} By the assumption, for any $\e>0$ there exists $k_0$ such that $\sum_{k/2<n\leq k}c_nu_n<\e$ for all $k>k_0$. But $\sum_{k/2 < n \leq k}c_nu_n \gg c_kku_k$, and the claim follows.
\end{proof}

\begin{lem}\label{lem2.1}
	Let $\{c_k\}$ be a sequence obeying (\ref{1.3}) and (\ref{1.4}). Then for any $\g$ and numbers $L>l$ we have 
$$\sum_{k=l}^L|c_k-c_{k+1}|k^{\g} \ll_{\g, A,B} c_ll^{\g} + \sum_{k=l}^Lc_kk^{\g-1}. 	
$$	
\end{lem}	
	 
\begin{proof} If $L\leq 4l-1$, then by the assumptions on $\{c_k\}$ we have
$$	
\sum_{k=l}^L|c_k-c_{k+1}|k^{\g} \ll l^{\g}(c_l+c_{2l}) \ll c_ll^{\g}	
$$	
(the constants in the proof of this lemma are allowed to depend on $\g,A,B$), and we are done. Thus we may assume that $2^u\leq l<2^{u+1}\leq 2^{v-1}\leq L < 2^v$, where $v\geq u+3$. For any positive integer $s$ we have
$$
\sum_{k=2^s}^{2^{s+1}-1}|c_k-c_{k+1}|k^{\g} \ll c_{2^s}2^{s\g} \ll c_{2^s}\sum_{k=2^{s-1}}^{2^s-1}k^{\g-1} \ll \sum_{k=2^{s-1}}^{2^s-1}c_kk^{\g-1}.
$$
Using this inequality, we get
\begin{multline*}
	\sum_{k=l}^L|c_k-c_{k+1}|k^{\g} \leq \sum_{k=l}^{2^{u+2}-1}|c_k-c_{k+1}|k^{\g}+\sum_{s=u+2}^{v-1}\sum_{k=2^s}^{2^{s+1}-1}|c_k-c_{k+1}|k^{\g} \\
	\ll c_ll^{\g} + \sum_{k=2^{u+1}}^{2^{v-1}-1}c_kk^{\g-1} \ll c_ll^{\g}+\sum_{k=l}^Lc_kk^{\g-1},
\end{multline*}
as desired.
\end{proof}

We will need the following estimates for exponential sums due to
van der Corput and Heath-Brown. We use the standard notation $e(y)=e^{2\pi iy}$.

\begin{lem}\label{lem2.2}
	Let $f\colon[u,v]\to\R$ be a function with $\theta\leq|f'(y)|\leq 1-\theta$ and $f''(y)\neq0$ for $y\in[u,v]$. Then
	$$\sum_{u\leq n<v}e(f(n)) \ll \theta^{-1}.
	$$
\end{lem}	

\begin{proof}
	See \cite{IK}, Corollary 8.11.
\end{proof}

\begin{lem}\label{lem2.3}
Let $r\geq3$, $M\geq1$, and $f\colon[M,2M]\to\R$ be a function with
$\L\leq f^{(r)}(y)\leq\eta\L$ for some $\Lambda>0$ and $\eta\geq1$. Then for any $\e>0$
$$
\sum_{M\leq n< 2M} e(f(n)) \ll_{\eta,r,\e} M^{1+\e}\left(\L^{1/r(r-1)}+M^{-1/r(r-1)}+ M^{-2/r(r-1)}\L^{-2/r^2(r-1)} \right). 
$$
\end{lem}

\begin{proof}
This is Theorem 1 of \cite{HB} with the function $f(n+M)$ defined on $[0,M]$.
\end{proof}

\begin{cor}\label{cor2.4}
Let $\a>0$ be not an integer. Define the function
$$c(\a)=\begin{cases}
\a, & \mbox{ if } \a\in(0,1);\\
\frac{2}{3(\a+1)(\a+2)}, & \mbox{ if } \a>1.
\end{cases}
$$
Let $0<a'<b'$. Then for all $x\in[a',b']$ and $k\geq1$ we have 
$$\sum_{n=1}^ke(n^{\a}x) \ll_{\a,a',b'} k^{1-c(\a)}. 
$$
\end{cor}

\begin{proof} Let $\a>0$, $\a\notin\N$, and $0<a'<b'$. Fix $x\in[a',b']$ and define $f(y)=xy^{\a}$.
It is enough to show that for any $M\geq1$
\begin{equation}\label{2.1}
 \sum_{M\leq n<2M} e(f(n)) \ll_{\a,a',b'} M^{1-c(\a)}.	
\end{equation}
We may suppose that $M$	is sufficiently large depending on $\a,a',b'$, since otherwise the result follows by taking the implied constant large enough. We consider two cases.

\smallskip

a) Let $\a\in(0,1)$. Since $f''(y)\neq0$ and $|f'(y)|=\a xy^{\a-1}$, we have $|f'(y)|\leq \a b'(2M)^{\a-1}\leq 1/2$ and $|f'(y)|\geq\theta:= \a a'M^{\a-1}$. Now (\ref{2.1}) with $c(\a)=\a$ follows from Lemma \ref{lem2.2}.
	
\smallskip	
	
b) Let $\a>1$. We apply Lemma \ref{lem2.3} with the function $-f$ and $r=\lfloor \a \rfloor+2 \geq3$; then
$$
-f^{(r)}(y)=-\a(\a-1)\ldots(\a-\lfloor \a \rfloor -1) xy^{\a-\lfloor \a \rfloor -2} > 0;
$$
clearly $-f^{(r)}(y)\asymp_{\a,a',b'} M^{\a-\lfloor \a\rfloor - 2}=:\Lambda_1$. We have $M^{-2} \leq \Lambda_1 \leq M^{-1}$ and so 
$$
\sum_{M\leq n<2M} e(f(n)) \ll_{\a,a',b',\e} M^{1+\e}\left(M^{-1/r(r-1)} + M^{-2/r(r-1)+4/r^2(r-1)}\right).
$$
Note that $\frac{2}{r(r-1)}-\frac{4}{r^2(r-1)}=\frac{2}{r(r-1)}\left(1-\frac2r\right) \geq \frac{2}{3r(r-1)}>\frac{2}{3(\a+1)(\a+2)}=:c(\a)$. Now (\ref{2.1}) follows by taking $\e$ small enough depending on $\a$. 

\smallskip	

This concludes the proof.
\end{proof}

\begin{lem}\label{lem2.5}
Let $f$ be a real function with $|f'(y)|\leq 1-\theta$ and $f''(y)\neq0$ on $[u,v]$. Then
$$\sum_{u < n < v} e(f(n)) = \int_u^v e(f(y))dy + O(\theta^{-1}).
$$
\end{lem}

\begin{proof}
See \cite{IK}, Lemma 8.8.
\end{proof}

\begin{lem}\label{lem2.6} 
Let $\a\in(0,1)$ and $x\in(0,1/10)$. Then for any $k\geq1$ 
$$
\sum_{n=1}^k\sin n^{\a}x \ll_{\a} k^{1-\a}x^{-1}.
$$ 
\end{lem}

\begin{proof}
Consider the function $f(y)=y^{\a}x/(2\pi)$ for $y\geq1$; then $|f'(y)|\leq \a y^{\a-1}x \leq 1/10$. By the previous lemma we get, taking the imaginary parts,
$$\sum_{n=1}^k \sin n^{\a}x = \int_1^k \sin y^{\a}x \, dy + O(1)= \frac{1}{\a x^{1/\alpha}}\int_x^{k^{\a}x} t^{1/\a-1}\sin t \, dt +O(1).
$$	
Using integration by parts it is easy to see that $\left|\int_0^ut^a\sin t\,dt\right| \leq 2u^a$ for any $a>0$ and $u>0$. Since $1/\a-1>0$, we thus have
$$\sum_{n=1}^k \sin n^{\a}x \ll_{\a} x^{-1/\a}(k^{\a}x)^{1/\a-1} + O(1) \ll k^{1-\a}x^{-1}, 
$$
as desired.
\end{proof}

\begin{lem}\label{lem2.7} 
	Let $\a>1$ and $x\in(0,x_0(\a))$, where $x_0(\a)$ is small enough. Define $L_0=\lfloor x^{-1/\a}\rfloor+1$ and $L_1=\lfloor (2x\a)^{-1/(\a-1)} \rfloor$. Then for any $k$ with $L_0\leq k < L_1$ we have
	$$
	\sum_{n=L_0}^k\sin n^{\a}x \ll_{\a} x^{-1/\a}.
	$$ 
\end{lem}

\begin{proof}
	Consider the function $f(y)=y^{\a}x/(2\pi)$ for $y\geq1$; then $|f'(y)|\leq \a y^{\a-1}x \leq 1/2$ for any $y\in[L_0,k]$. By Lemma \ref{lem2.5} we get, taking the imaginary parts,
	$$\sum_{n=L_0}^k \sin n^{\a}x = \int_{L_0}^k \sin y^{\a}x \, dy + O(1)= \frac{1}{\a x^{1/\alpha}}\int_{L_0^{\a}x}^{k^{\a}x} t^{1/\a-1}\sin t \, dt +O(1).
	$$	
	Using integration by parts it is easy to see that $\left|\int_u^vt^a\sin t\,dt\right| \ll u^a$ for any $a<0$ and $0<u<v$. Since $1/\a-1<0$, we thus have
	$$\sum_{n=1}^k \sin n^{\a}x \ll_{\a} x^{-1/\a}(L_0^{\a}x)^{1/\a-1} + O(1) \ll x^{-1/\a}, 
	$$
	as desired.
\end{proof}

We will also need another estimate for exponential sums.

\begin{lem}[H.Weyl]\label{lem2.8}
	Let $M\geq2$, $r\geq2$ and let a function $f\colon [M,2M]\to \R$ be such that
	$$\frac{F}{B} \leq \frac{y^r}{r!}|f^{(r)}(y)| \leq F
	$$ 	
	for all $y\in [M,2M]$, where $B\geq 1$ and $F>0$. Then for any $1\leq M'\leq M$
	$$\sum_{M\leq n\leq M+M'} e(f(n)) \ll 
	B^{2^{2-r}}\left( FM^{-r} + F^{-1}\right)^{2^{2-r}r^{-1}} M\log M.
	$$ 
\end{lem}

\begin{proof}
	See \cite{IK}, Theorem 8.4.
\end{proof}

\begin{lem}\label{lem2.9}
	Let $\a>0$ be not an integer, $x\in(0,1)$, and $M\in\N$ be such that $M\geq x^{-1/(\a-\d)}$ for some $\d\in(0,\a)$. Then there exists absolute constant $d(\a,\d)\in(0,1)$ such that for any $1\leq M' \leq M$ 
	\begin{equation*}
	\sum_{M\leq n\leq M+M'}e(n^{\a}x)\ll_{\a} M^{1-d(\a,\d)}.
	\end{equation*}	
\end{lem} 

\begin{proof} 
	We apply Lemma \ref{lem2.8} with $f(y)=xy^{\a}$, $r=\lfloor\a\rfloor+2$ (say), $F=C_1xM^{\a}$, and $B=C_2$ for appropriate constants $C_1=C_1(\a)$ and $C_2=C_2(\a)$. Since $x^{-1}\leq M^{\a-\d}$, we then obtain 
	$$
	\sum_{M\leq n\leq M+M'}e(n^{\a}x)\ll_{\a}\left( M^{-1}+x^{-1}M^{-\a}\right)^{2^{2-r}r^{-1}} M\log M \ll M^{1-d(\a,\d)}
	$$
	for some $d(\a,\d)\in(0,1)$, as desired.
\end{proof}

\section{Proof of Theorem \ref{th1.1}} \label{sec3}

\subsection{Proof of the first assertion of Theorem \ref{th1.1}} 

Fix $\a>0$, $\a\notin\N$, and let $c(\a)$ be defined as in Corollary \ref{cor2.4}. Let $\{c_k\}$ be a sequence which obeys (\ref{1.3}), (\ref{1.4}), and $\sum_kc_kk^{-c(\a)}<\infty$. Fix arbitrary positive numbers $a<b$; we will prove that the series (\ref{1.1}) and (\ref{1.2}) converge uniformly on $[a,b]$. To do this, it is enough to show that for any $\e>0$ there exists $l=l(\e)\in\N$ such that for all $L>l$ we have
\begin{equation}\label{3.1}
\sup_{x\in[a',b']} \left|\sum_{k=l}^Lc_ke(k^{\a}x)\right|<\e,
\end{equation}
where $a'=a/(2\pi)$, $b'=b/(2\pi)$. Let $b_l=\sup_{k\geq l}c_kk^{1-c(\a)}$ and $V_k(x)=\sum_{n=1}^ke(n^{\a}x)$. Then using consequently Abel's summation, Corollary \ref{cor2.4}, and Lemma \ref{lem2.1}, for all $x\in[a',b']$ we have (here the implied constants in $\ll$ are allowed to depend on $\a$, $a$, and $b$)
\begin{multline*}
\left|\sum_{k=l}^Lc_ke(k^{\a}x)\right| = \left|\sum_{k=l}^{L-1}(c_k-c_{k+1})V_k(x) - c_lV_{l-1}(x)+c_LV_L(x) \right|  \\
\ll b_l+\sum_{k=l}^{L-1}|c_k-c_{k+1}|k^{1-c(\a)} \ll
 b_l + \sum_{k\geq l} c_kk^{-c(\a)}. 
\end{multline*}
Now (\ref{3.1}) follows from Lemma \ref{lem2.0} by taking $l$ sufficiently large depending on $\a$, $a$, $b$, and $\e$. This concludes the proof.

\subsection{Proof of the second assertion of Theorem \ref{th1.1}}

We first prove the claim for the series (\ref{1.1}). Recall that a sequence $\{ a_n\}$ is said to be uniformly distributed modulo $1$ if for any fixed interval $[c,d]\subseteq [0,1]$ 
$$\#\{1\leq n\leq N: a_n \pmod 1 \in (c,d) \} =(d-c+o(1))N, \quad N\to\infty.
$$
Let $\a\in(0,1)$. Fix any $x\neq0$; we may suppose that $x>0$. It is well-known (see, for example, \cite{Mur}, Exercise 11.6.3) that the sequence $\{\sigma n^{\a}\}$ is uniformly distributed modulo $1$ for any $\sigma\neq0$; we take $\sigma=x/(2\pi)$. Now let $m$ be large enough depending on $\a$ and $x$; it follows that there exists $n=n(m)\in[m,2m]$ such that
$$ \frac{n^{\a}x}{2\pi} \in \left(\frac18,\frac14\right) \pmod 1.
$$
Since for any $r>0$ we have
$$ (n+r)^{\a} - n^{\a} = n^{\a}\left( (1+r/n)^{\a} - 1\right) = n^{\a}(\a r/n +O(r^2/n^2))=\frac{\a r}{n^{1-\a}}+O\left(\frac{r^2}{n^{2-\a}}\right),
$$
we see that
$$\frac{(n+r)^{\a}x}{2\pi} \in \left(\frac18,\frac38\right) \pmod 1
$$
for all $r=0,1,\ldots, r_0 = \lfloor n^{1-\a}\frac{\pi}{8\a x} \rfloor$ (say). It means that for these $r$ we have
\begin{equation}\label{3.2}
(n+r)^{\a}x \in   \left(\frac{\pi}{4},\frac{3\pi}{4}\right) \pmod {2\pi}.
\end{equation} 
On the other hand, if the series $\sum_kc_k\sin k^{\a}x$ converges, then for any $\e>0$ there exists $m_0$ such that for all $m_2>m_1>m_0$ 
$$\left|\sum_{k=m_1}^{m_2}c_k\sin k^{\a}x\right| <\e .
$$
Now we take any $m>m_0$ (again large enough depending on $x$ and $\a$) and set $m_1=n=n(m)$, $m_2=n+r_0$; then by (\ref{1.3}) and (\ref{3.2}) we get
$$\e \gg \sum_{k=n}^{n+r_0}c_k \gg c_{3m}r_0 \gg c_{3m}m^{1-\a},
$$
since $m\gg_{\a,x}1$. It follows that $c_mm^{1-\a}\ll\e$ whenever $m$ is large enough, and the claim follows.

The proof for the case of the series (\ref{1.2}) is completely similar; we just need to consider the interval $(-1/8,0)$ (say) instead of $(1/8,1/4)$, and then we get (\ref{3.2}) with the interval $(\frac{\pi}{4}, \frac{3\pi}{4})$ replaced by $(-\frac{\pi}{4}, \frac{\pi}{4})$. This concludes the proof of Theorem {\ref{1.1}}.

\section{Proof of Theorem \ref{th1.2}}\label{sec4}

We first show that the condition is necessary. Let $\a>0$ be arbitrary, and suppose that the series (\ref{1.1}) converges uniformly on $(0,a)$ for some $a>0$. Then for any $\e>0$ there exists $l=l(\e)$ such that for all $L>l$ 
$$
\sup_{x\in(0,a)}\left|\sum_{k=l}^Lc_k\sin k^{\a}x\right|<\e.
$$
We can assume that $l$ is large enough depending on $\a$. Take $L=\lfloor 2^{1/\a}l\rfloor$ and $x=\pi l^{-\a}/4\in(0,a)$; then $L\ll _{\a}l$ and $c_Ll\ll_{\a,A} \sum_{k=l}^{L}c_k \ll \e$, and hence $c_kk\to0$, $k\to\infty$, as desired.

Now we prove that the condition is sufficient. Fix noninteger $\a>0$, and let $c_kk=o(1)$, $k\to\infty$. We need to show that  the series (\ref{1.1}) converges uniformly on any bounded subset of~$\R$. In view of Theorem \ref{th1.1} we may restrict our attention to the interval $(0,x_0(\a))$, where $x_0(\a)$ is sufficiently small depending on $\a$; it is enough to show that
\begin{equation}\label{4.1}
\sup_{L>l}\sup_{x\in(0,x_0(\a))} \left|\sum_{k=l}^L c_k\sin k^{\a}x \right| \ll b_l,
\end{equation}
where $b_l=\sup_{k\geq l}c_kk$ (in this section the implied constants are allowed to depend on $\a$). Fix arbitrary $x\in(0,x_0(\a))$. There are two cases to consider. 

\medskip 

a) Let $\a\in(0,1)$. Define $L_0=\lfloor x^{-1/\a}\rfloor+1$. First, we can use the bound
\begin{equation}\label{4.2}
\left|\sum_{k=l}^{\min\{L,L_0\}}c_k\sin k^{\a}x\right| \leq \sum_{k=l}^{L_0}c_kk^{\a}x \leq b_lx\sum_{k\leq L_0}k^{\a-1} \ll b_lxL_0^{\a} \ll b_l 
\end{equation} 
to reduce to the case $l>L_0$. Define $S_k(x)=\sum_{n=1}^k\sin k^{\a}x$. Applying Lemma \ref{lem2.6} and Lemma \ref{lem2.1}, for any $l>L_0$ we have 
\begin{multline*}
\left|\sum_{k=l}^Lc_k\sin k^{\a}x\right| \ll b_l +\sum_{k=l}^{L-1}|c_k-c_{k+1}||S_k(x)| \ll b_l + c_ll^{1-\a}x^{-1} +\sum_{k=l}^{L}c_kk^{-\a}x^{-1}\\
\ll b_l\left(1+ L_0^{-\a}x^{-1}+\sum_{k=l}^{L}k^{-\a-1}x^{-1} \right) \ll b_l\left(1+ L_0^{-\a}x^{-1} \right)\ll b_l, 
\end{multline*}
as desired.   

\medskip

b) Let $\a>1$. Define $L_0=\lfloor x^{-1/\a}\rfloor+1$ and $L_1=\lfloor (2x\a)^{-1/(\a-1)} \rfloor$. Due to the bound (\ref{4.2}) we again may suppose that $l>L_0$. Suppose that $l \leq L_1$;  For $k\geq L_0$, set $\overline{S}_k(x)=\sum_{n=L_0}^k\sin n^{\a}x$. By Lemmas \ref{lem2.7} and \ref{lem2.1} we have
\begin{multline*}\label{4.4}
\left|\sum_{k=l}^{\min\{L,L_1\}}c_k\sin k^{\a}x\right| \ll b_l +\sum_{k=l}^{L_1-1}|c_k-c_{k+1}||\overline{S}_k(x)| \ll b_l +\sum_{k=l}^{L_1}|c_k-c_{k+1}|x^{-1/\a}\ll \\ b_l\left(1+l^{-1}x^{-1/\a}+\sum_{k=l}^{L_1}k^{-2}x^{-1/{\a}}\right) \ll b_l\left(1+l^{-1}x^{-1/\a}\right) \ll b_l(1+L_0^{-1}x^{-1/\a}) \ll b_l.
\end{multline*}
Thus we may assume that $l>L_1$, and so
$$
\left|\sum_{k=l}^Lc_k\sin k^{\a}x\right| \ll b_l +\sum_{k=l}^{L}|c_k-c_{k+1}||\widetilde{S}_k(x)|,
$$
where $\widetilde{S}_k(x)=\sum_{n=l}^k\sin k^{\a}x$. Lemma \ref{lem2.9} implies that $\widetilde{S}_k(x) \ll k^{1-d(\a)}$ for some $d(\a)>0$. As before, Lemma \ref{lem2.1} gives us
$$
\left|\sum_{k=l}^Lc_k\sin k^{\a}x\right| \ll b_l +\sum_{k=l}^{L}c_kk^{-d(\a)}\ll b_l\left(1+ \sum_{k=L_1}^{\infty}k^{-d(\a)-1} \right) \ll b_l.
$$

\smallskip

Thus, (\ref{4.1}) follows for any $x\in(0,x_0(\a))$ and $L>l$. This completes the proof of Theorem \ref{th1.2}.

\section{Proof of Theorem \ref{th1.3}} \label{sec5} 

We need the following two results. 

\begin{theorem}[\cite{Bes}]\label{th5.1}
	Let $p_1, \ldots, p_s$ be distinct primes, and $b_1,\ldots,b_s$ be positive integers not divisible by any of these primes. Let, further, $x_1,\ldots,x_s$ be positive real roots of the equations
$$x^{n_1}-p_1b_1=0, \, \ldots, \, x^{n_s}-p_sb_s=0,
$$	
respectively, and $P(y_1,...,y_s)$ be a nonzero polynomial with rational coefficients of degree at most $n_i-1$ with respect to $y_i$ for each $i=1,\ldots,s$. Then	$P(x_1,...,x_s)\neq0$.
\end{theorem}

For a number $a\in\R$, we use the standard notation $\|a\|=\min_{m\in\Z}|a-m|$. 

\begin{lem}\label{lem5.2}
Let $\overline{\a}=(1,\a_1,\ldots,\a_{\nu})\in \R^{\nu+1}$ be a vector whose components are linearly independent over $\Q$. Then for any $\d\in(0,1/2)$ and $\beta_1,\ldots,\beta_{\nu}\in\R$ there are infinitely many integers $x$ with
$$\|x\a_j+\beta_j\|<\d, \qquad j=1,\ldots,\nu. 
$$
\end{lem}

This well-known lemma easily follows from the multidimensional Weyl's criterion of equidistribution modulo $1$ (see, for instance, \cite{KN}, Theorem 6.2). We also note that in \cite{Kor} (see Corollary 3.4) a stronger result was proved: for any $\d\in(0,1/2)$ there exists $c=c(\overline{\a},\d)>0$ such that for any $\beta_1,\ldots,\beta_{\nu}\in\R$ each interval of length $c$ contains a number $x$ with $\|x\a_j+\beta_j\|<\d$, \, $j=1,\ldots,\nu$.

\medskip 

Fix a noninteger rational number $\a>0$. Let $n_1,\ldots,n_{\nu}$ be square-free integers up to a positive integer $L$. We first show that the numbers $n_j^{\a}$ are linearly independent over $\Q$. If we write $\a=u/v$, where $u,v\in\N$, $u=qv+r$, and $0<r<v$, then  clearly it is enough to show that the numbers $n_j^{r/v}$ are linearly independent over $\Q$. Let $p_1,\ldots, p_s$ be primes up to $L$. Then any linear combination of the numbers $n_1^{r/v},\ldots, n_{\nu}^{r/v}$ with rational coefficients can be thought of as the value of an appropriate polynomial $P(x_1,\ldots,x_s)$ at the point $x_1=p_1^{1/v}, \ldots, x_s=p_s^{1/v}$; for this $P$ we have $\deg_{x_i}P=r<v$ for each $i=1,\ldots,s$. Applying Theorem \ref{th5.1} with $b_i=1$ and $n_i=v$ for all $i=1,\ldots,s$, we see that this value can be equal to zero only if so are all the coefficients of our linear combination. Thus, the numbers $n_{j}^{\a}$ are linearly independent over $\Q$.

Now we show that there exists $x_0=x_0(L,\a)\in\R$ such that
\begin{equation}\label{5.1}
\sin n^{\a}x_0\geq 0.99 
\end{equation}
for all square-free $n\leq L$. To do this, note that the numbers $1,n_1^{\a}/(2\pi),...,n_{\nu}^{\a}/(2\pi)$ are also linear independent over $\Q$ (since $2\pi$ is transcendental, it immediately follows from the independence of $n_j^{\a}$); applying Lemma \ref{lem5.2} with $\d=10^{-2}$ and $\beta_i=-1/4$, we see that there exist $x_0=x_0(L,\a)\in\R$ and integers $m_j$ with
$$ | n_j^{\a}x_0/(2\pi) - m_j -1/4 |  <  10^{-2}.
$$
Then
\begin{equation*}
|n_j^{\a}x_0 - 2\pi m_j - \pi/2|<\pi/50;
\end{equation*}
so, we have $n_j^{\a}x_0\in (\pi/2-\pi/50,\pi/2+\pi/50) \pmod{2\pi}$, and (\ref{5.1}) follows. 

Now we are ready to prove Theorem \ref{th1.3}. Trivially, the condition $\sum_kc_k<\infty$ is sufficient for the uniform convergence of the series (\ref{1.1}) on $\R$, so we prove the converse. This uniform convergence is equivalent to the following: for any $\e>0$ there exists $l=l(\e)\in\N$ such that for all $L>l$ 
we have 
\begin{equation}\label{5.2}
\sup_{x\in\R}\left|\sum_{k=l}^Lc_k\sin k^{\a}x\right| < \e,
\end{equation}
and we may suppose that $l$ is large enough. Fix $L$. Using Abel's summation, we can write
\begin{equation}\label{5.3}
\sum_{k=l}^Lc_k\sin k^{\a}x = \sum_{k=l}^{L-1}(c_k-c_{k+1})S_k(x) - c_lS_{l-1}+c_LS_L, 
\end{equation}
where $S_k(x)=\sum_{n=1}^k\sin n^{\a}x$. It is well-known that the number of square-free positive integers up to $k$ is $\frac{6}{\pi^2}k+O(k^{1/2})$, and $\frac{6}{\pi^2}=0.607...$. Let $A(n)$ be the indicator function of square-free integers. Then for $x_0=x_0(L,\a)$ and all $k$ with $l\leq k\leq L$, we have by (\ref{5.1}) 
\begin{equation*}
S_k(x_0)=\sum_{n\leq k} A(n)\sin n^{\a}x_0 + \sum_{n\leq k} (1-A(n))\sin n^{\a}x_0 \geq 0.99\cdot0.6k-0.4k > 0.1k,
\end{equation*}
since $l$ (and hence $k$) is large enough. From here, (\ref{5.3}), and monotonicity of $c_k$, we have
$$\sum_{k=l}^Lc_k\sin k^{\a}x_0 > 0.1\sum_{k=l}^{L-1}(c_k-c_{k+1})k - c_ll + 0.1c_LL= 0.1\sum_{k=l+1}^Lc_k  - 0.9c_ll. 
$$
This bound implies that the series $\sum_kc_k$ converges, since otherwise $\sum_{k=l}^Lc_k\sin k^{\a}x_0$ could be made arbitrarily large by choice of $L$, and it would contradict (\ref{5.2}). This concludes the proof.

\end{document}